\documentclass[11pt,a4paper]{amsart}

\usepackage{amssymb}
\usepackage{amsmath}
\usepackage{amsfonts}
\usepackage[english]{babel}
\usepackage[T1]{fontenc}
\usepackage[latin1]{inputenc}
\usepackage{amsthm}
\usepackage[all]{xy}

\date{\today}
\theoremstyle{definition}

\numberwithin{equation}{section}

\theoremstyle{plain}
\newtheorem{theorem}{Theorem}[section]
\newtheorem{proposition}[theorem]{Proposition}
\newtheorem{corollary}[theorem]{Corollary}
\newtheorem{lemma}[theorem]{Lemma}

\theoremstyle{remark}
\newtheorem{remarks}[theorem]{Remarks}

\newtheorem{example}[theorem]{Example}
\newtheorem{problems}[theorem]{Problems}
\newtheorem{clm}[theorem]{Claim}

\newcommand{\BMOA}{\mathit{BMOA}}
\newcommand{\VMOA}{\mathit{VMOA}}
\newcommand{\UnitDisk}{\mathbb{D}}

\newcommand{\T}{\mathbb{T}}
\renewcommand{\Re}{\operatorname{Re}}

\title[Rigidity of composition operators]{Rigidity of composition operators on 
the Hardy space $H^p$}
\author{Jussi Laitila}
\address{Jussi Laitila, Department of Mathematics and Statistics, P.O.\ Box 68, FI-00014 University of Helsinki, Helsinki, Finland}
\email{jussi.laitila@helsinki.fi}
\author{Pekka J. Nieminen}
\address{Pekka J. Nieminen, Department of Mathematics and Statistics, P.O.\ Box 68, FI-00014 University of Helsinki, Helsinki, Finland}
\email{pjniemin@cc.helsinki.fi}
\author{Eero Saksman}
\address{Eero Saksman, Department of Mathematics and Statistics, P.O.\ Box 68, FI-00014 University of Helsinki, Helsinki, Finland}
\email{eero.saksman@helsinki.fi}
\author{Hans-Olav Tylli}
\address{Hans-Olav Tylli, Department of Mathematics and Statistics, P.O.\ Box 68, FI-00014 University of Helsinki, Helsinki, Finland}
\email{hans-olav.tylli@helsinki.fi}

\begin{document}

\begin{abstract}

Let $\phi$ be an analytic map taking the unit disk $\UnitDisk$  into itself. 
We establish that the class of composition operators
$f \mapsto C_\phi(f) = f \circ \phi$ exhibits a rather strong rigidity of
non-compact behaviour on the Hardy space $H^p$, for $1\le p < \infty$ 
and $p \neq 2$. Our main result is the following trichotomy, which states that 
exactly one of the following alternatives holds: 
 (i)  $C_\phi$ is a compact operator $H^p \to H^p$,
 (ii) $C_\phi$ fixes a (linearly isomorphic) copy of $\ell^p$ in $H^p$, but  
 $C_\phi$ does not fix any copies of $\ell^2$ in $H^p$, 
 (iii) $C_\phi$ fixes a copy of $\ell^2$ in $H^p$. 
Moreover, in case (iii) the operator $C_\phi$ 
actually fixes a copy of $L^p(0,1)$ in $H^p$ provided $p > 1$.
We reinterpret these results in terms of 
 norm-closed ideals of the bounded linear operators on $H^p$, which 
 contain the compact operators $\mathcal K(H^p)$.
In particular, the class of composition operators on $H^p$ does not reflect  
 the quite complicated lattice structure of such ideals.
\end{abstract}

\maketitle

\section{Introduction and preliminaries}\label{intro}

Let $\mathbb D = \{z \in \mathbb C: \vert z\vert < 1\}$ be the unit disk in 
$\mathbb C$. For $0 <  p < \infty$ the analytic function 
$f: \mathbb D \to \mathbb C$ belongs to the Hardy space $H^p$ if
 \begin{equation}\label{Hnorm}
 \Vert f\Vert^p_{p} = \sup_{0 \le r < 1} \int_{\mathbb  T}  \vert f(r\xi)\vert^p dm(\xi) < \infty,
 \end{equation}
where $\mathbb T = \partial \mathbb D$ 
(identified with $[0,2\pi]$) and  $dm(e^{it}) = \frac{dt}{2\pi}$. 
Let  $\phi:  \UnitDisk \to \UnitDisk$ be an analytic self-map of $\UnitDisk$. 
It is a well-known consequence of the Littlewood subordination principle, 
see e.g. \cite[3.1]{CM}, that   the  composition operator  
\[
f \mapsto C_\phi(f) = f \circ \phi
\]
is bounded $H^p \to H^p$ for any $\phi$ as above. 
Properties of these composition operators have been studied very extensively during the last 40 years on various Banach spaces of analytic functions on $\mathbb D$,
see  \cite{CM} and  \cite{Sh2} for comprehensive expositions of the early developments of the 
area. The compactness of  $C_\phi$ on $H^p$ is well understood, and there are several 
equivalent characterisations in the literature. 
To exhibit a specific criterion recall that 
 Shapiro \cite{Sh1} established that $C_\phi$ is a compact operator  
$H^p \to H^p$  if and only if  
\begin{equation}\label{Sh}
\lim_{\vert w \vert\to 1} \frac{N(\phi,w)}{\log(1/\vert w\vert)} = 0.
\end{equation}
Above $N(\phi,w)$ is the Nevanlinna counting function of $\phi$ defined by 
$N(\phi,w) = \sum_{z \in \phi^{-1}(w)} \log(1/\vert z\vert)$  for 
$w \in \phi(\mathbb D)$ (counting multiplicities). 
Finer gradations of compactness  were obtained e.g. 
by Luecking and Zhu \cite{LZ},  who characterised the 
membership of $C_\phi$ in the Schatten $p$-classes on $H^2$.
Moreover,  the approximation numbers of $C_\phi$ on $H^2$ were estimated in 
\cite{LLQR}, \cite{LQR1} and \cite{LQR2}, 
as well as on $H^p$ in \cite{LQR3}.

The purpose of this paper is to demonstrate that  composition operators on $H^p$ 
only allow a small variety of qualitative non-compact  behaviour compared 
to that of arbitrary bounded operators on $H^p$. Let  $E, F$ and $X$ be Banach spaces.
It will be convenient to say that the bounded linear operator 
$U: E \to F$  \textit{fixes a copy of} $X$ in $E$ if there is an infinite-dimensional subspace $M \subset E$,  $M$ linearly isomorphic to 
$X$,  for which  $U_{|M}$ is bounded below on $M$, that is, 
there is $c > 0$ so that 
$\Vert Ux\Vert \ge c \cdot \Vert x \Vert$ for all $x \in M$.
We use the standard notation $M  \approx X$ for linearly isomorphic
spaces $M$ and $X$, and refer to \cite{AK}, \cite{LT} and \cite{Wo} 
for general background related to the theory of  Banach spaces.

The trichotomy contained in Theorem \ref{Main0} below is the main result of this paper. 
Let $E_\phi = \{ e^{i\theta} : |\phi(e^{i\theta})| = 1\}$ 
be the boundary contact set of  the analytic map $\phi: \mathbb D \to \mathbb D$.
Here, and in the sequel, we use $\phi(e^{i\theta})$ to denote 
 the a.e. radial limit function of $\phi$ on $\mathbb T$. 
It is part of the trichotomy that \eqref{Sh} together with the simple condition 
\begin{equation}\label{CM}
m(E_\phi) = 0
\end{equation}
completely determine the composition operators which fix copies of the subspace $\ell^p$ or 
$\ell^2$ 
in  $H^p$.  Recall that the known compactness  results for $C_\phi$ on $H^2$
yield that  \eqref{Sh} implies \eqref{CM}, but 
the class of symbols $\phi$ satisfying \eqref{CM} is much larger 
than that of \eqref{Sh}, see e.g. \cite[Chap. 10]{Sh2}.

In the statement  below we exclude the Hilbert space $H^2$, 
where the situation is known and much simpler, 
since part (ii)  does not occur for $p = 2$ 
(cf. the discussion following Theorem \ref{Main1}).
We use  $\mathcal K(E)$ to denote
the class of compact operators $E \to E$ for any Banach space $E$, and
take into account the known characterisation of 
the composition operators $C_\phi \in \mathcal K(H^p)$.

\begin{theorem}\label{Main0}
Let $1\leq  p < \infty$,  $p \neq 2$, and $\phi$ be any analytic self-map of $\UnitDisk$. Then there are three mutually exclusive alternatives:
\begin{itemize}
\item[(i)]  $C_\phi$ is compact on $H^p$,
\item[(ii)]  $C_\phi$ fixes  a copy of $\ell^p$ in $H^p$, but 
does not fix any copies of $\ell^2$ in $H^p$,
\item[(iii)]  $C_\phi$ fixes a copy of $\ell^2$ (as well as of $\ell^p$) in $H^p$. In this case, if 
$1<  p < \infty$ and $p \neq 2$, then 
$C_\phi$ also fixes a copy of $L^p(0,1)$ in $H^p$.
\end{itemize}
Furthermore, regarding the above alternatives
\begin{itemize}
\item[(i)]  takes place if and only if  Shapiro's condition \eqref{Sh} holds,
\item[(ii)]  takes place if and only if  \eqref{Sh} fails to hold but $m(E_\phi) = 0$,
\item[(iii)]  takes place if and only if $m(E_\phi) > 0$.
\end{itemize}
In particular, $C_\phi \in \mathcal K(H^p)$ if and only if $C_\phi$ does not fix any copies of 
$\ell^p$ in $H^p$.
\end{theorem}

Theorem \ref{Main0} is obtained by combining Theorems \ref{Main1}, \ref{Main2} and \ref{Main3} 
stated below, which also contain more precise information. 
For this purpose we first recall  some 
standard linear classes that classify the behaviour of non-compact  
operators.  Let  $E$, $F$ and $X$ be Banach spaces, and $\mathcal L(E,F)$
be the space of bounded linear operators from $E$ to $F$.
The operator  $U \in \mathcal L(E,F)$ is called \textit{$X$-singular} if $U$ does not fix 
any  copies of $X$ in $E$. We denote 
\begin{align*}
\mathcal S_X(E,F) =  \{U \in \mathcal L(E,F): U\ \textrm{is}\ X\textrm{-singular}\}, 
\end{align*}
and put $\mathcal S_p(E,F) = \mathcal S_{\ell^p}(E,F)$ to simplify our notation in the case of
 $X = \ell^p$.  Recall further that
$U \in \mathcal L(E,F)$ is   \textit{strictly singular}, 
denoted by $U \in \mathcal S(E,F)$, if $U$ is not bounded below on 
any infinite-dimensional linear subspaces  $M \subset E$. 
It is clear that $\mathcal K(E,F) \subset \mathcal S(E,F) \subset \mathcal S_p(E,F)$
for any Banach spaces $E$ and $F$, and it is known that the 
classes   $\mathcal S(E,F)$ and   $\mathcal S_p(E,F)$ 
define norm-closed operator ideals in the sense of Pietsch 
\cite{P} for any $1 \le p \le \infty$
(cf. \cite[p. 289]{W} for the case of $\mathcal S_p$). 

Part of Theorem \ref{Main0} is contained in the following dichotomy, 
which we also relate to the known characterisation of the 
compact composition operators on $H^p$.

\begin{theorem}\label{Main1}
Let $1\le  p < \infty$ and let $\phi:  \UnitDisk \to \UnitDisk$
be any analytic map. 
Then either $C_\phi \in \mathcal K(H^p)$, or else 
$C_\phi \notin \mathcal S_p(H^p)$. Equivalently,  $C_\phi$ fixes a copy of $\ell^p$
in $H^p$ if and only if \eqref{Sh} does not hold.
\end{theorem}

The above theorem holds for $p = 2$ 
because of the general fact due to Calkin that 
$\mathcal K(H^2) = \mathcal S(H^2) = \mathcal S_2(H^2)$
for the Hilbert space $H^2$, see e.g. \cite[5.1-5.2]{P}.
For $1 < p < \infty$ and $p \neq 2$ one has that
\begin{equation}\label{Weis}
\mathcal S(H^p) = \mathcal S_2(H^p) \cap \mathcal S_p(H^p).
\end{equation}
This follows from the 
characterisation of $\mathcal S(L^p)$ by Weis  \cite{W}  combined 
with the well-known fact   that  $H^p \approx L^p \equiv L^p(0,1)$, 
see e.g. \cite[2.c.17]{LT2}.
By contrast, for $p \neq 2$ all the inclusions 
\begin{equation}\label{strict}
\mathcal K(H^p) \varsubsetneq \mathcal S(H^p), \quad
\mathcal S(H^p) \varsubsetneq \mathcal S_2(H^p), \quad
\mathcal S(H^p) \varsubsetneq \mathcal S_p(H^p)
\end{equation}
 are strict. This is easily deduced from the facts that 
$H^p \approx L^p$ contains complemented subspaces isomorphic to  
$\ell^p$ and $\ell^2$, whereas any $U \in \mathcal L(\ell^p,\ell^q)$
is strictly singular for $p \neq q$, see e.g. \cite[2.c.3]{LT}. 
Thus  Theorem \ref{Main1} states  that for $p \neq 2$  the compactness 
of composition operators 
$C_\phi \in \mathcal L(H^p)$ is a fairly rigid property as compared to \eqref{Weis}
and \eqref{strict} for arbitrary operators.  
It is also convenient to rephrase this as follows:

\begin{corollary}\label{main}
For $1 \le p < \infty$ the following conditions are equivalent for 
any analytic map  $\phi:  \UnitDisk \to \UnitDisk$:
\begin{itemize} 
\item[(i)] $\phi$ satisfies \eqref{Sh}, 
\item[(ii)] $C_\phi \in \mathcal K(H^p)$,
\item[(iii)]  $C_\phi \in \mathcal S(H^p)$,
\item[(iv)]  $C_\phi \in \mathcal S_p(H^p)$.
\end{itemize}
\end{corollary}

The first result (excluding the case $H^2$) in the direction of Theorem \ref{Main1}
and Corollary \ref{main}
 is due to Sarason \cite{S1},  who showed that  $C_\phi$ 
is weakly compact $H^1 \to H^1$ if and only if it is compact. 
Jarchow \cite[p. 95]{J} pointed out that as a consequence 
$C_\phi \in \mathcal K(H^1)$ if and only if 
$C_\phi$ is weakly conditionally compact on $H^1$, that is, 
$C_\phi \in \mathcal S_1(H^1)$ in view of Rosenthal's $\ell^1$-theorem, see e.g. \cite[2.e.5]{LT}.
Hence the case $p = 1$ in Theorem \ref{Main1} and Corollary \ref{main} was known 
earlier.  We refer to Subsection \ref{rigidity}
for a list of further references to analogous rigidity results for composition operators on several 
(classical) Banach spaces $E$ of analytic functions on the unit disk $\UnitDisk$.

The lattice structure of the operator norm-closed ideals of 
$\mathcal L(H^p) \approx \mathcal L(L^p)$ 
containing the compact operators is quite complicated for 
$1 < p < \infty$ and $p \neq 2$, see e.g. \cite[5.3.9]{P} and \cite{Sch}. For instance,
$\mathcal S_p(H^p)$ and $\mathcal S_2(H^p)$ are mutually incomparable classes, since
 $H^p \approx L^p$ contains complemented copies of  $\ell^2$ and $\ell^p$.
However, note that Corollary \ref{main} implies that  if 
$C_\phi \in \mathcal L( H^p)$ fixes a copy of 
$\ell^2$ in $H^p$,  then  $C_\phi$ must also fix a copy of $\ell^p$ in $H^p$. 
These facts raise the problem whether it is possible to explicitly 
determine the $\ell^2$-singular 
composition operators on $H^p$. 
In turns out in Theorem \ref{Main2} below that condition \eqref{CM}  
characterises this class, thus providing a finer 
classification of the non-compact $C_\phi \in \mathcal L(H^p)$
for $1 \le p < \infty$ and $p \neq 2$. 
We stress that Theorem \ref{Main2} (as well as 
the subsequent Theorem \ref{Main3})  does \textit{not} hold for $H^2$. 

\begin{theorem}\label{Main2}
Let $1 \le p < \infty$, $p \neq 2$,  and $\phi:  \UnitDisk \to \UnitDisk$
be an analytic map. Then  $C_\phi$ fixes a copy of $\ell^2$
in $H^p$ if and only if $m(E_\phi) > 0$. Equivalently,
$C_\phi \in \mathcal S_2(H^p)$ if and only if  \eqref{CM} holds.
\end{theorem}

Cima and Matheson \cite{CiM1} have shown that  (\ref{CM})
characterises the completely continuous  composition operators  
$C_\phi \in \mathcal L( H^1)$. 
As a significant strengthening of Theorem \ref{Main2} we are further able to
show that for $p > 1$ (and $p \neq 2$) 
condition (\ref{CM}) actually  describes the operators $C_\phi$ which belong to the class 
$\mathcal  S_{L^{p}}(H^p)$. Here $\mathcal S_{L^{p}}(H^p)$ 
is the maximal non-trivial ideal of $\mathcal L(H^p)$, see \cite[p. 103]{DJS}.
To state the relevant result let $h^p$ be the harmonic Hardy space consisting 
of the harmonic functions 
$f:  \UnitDisk \to \mathbb C$ normed by \eqref{Hnorm}.

\begin{theorem}\label{Main3}
Let $1 < p < \infty$, $p \neq 2$,  and  $\phi:  \UnitDisk \to \UnitDisk$
be an analytic map. Then  the following conditions are equivalent:
\begin{itemize}
\item[(i)] $\phi$ satisfies $m(E_\phi) = 0$, 
\item[(ii)] $C_\phi \in \mathcal  S_{L^{p}}(H^p)$, that is, $C_\phi$ does not fix any copies 
of $L^p$ in $H^p$,
\item[(iii)] $C_\phi \in \mathcal  S_{L^{p}}(h^p)$, 
\item[(iv)] $C_\phi \in \mathcal S_2(H^p)$.
\end{itemize}
\end{theorem}

The paper is organised as follows. 
The proof of Theorem \ref{Main1} is  given in Section \ref{Proof}.
The argument is based on explicit perturbation estimates, where the starting point is 
a known test function reformulation of the compactness criterion \eqref{Sh}.
The proofs of Theorems \ref{Main2} and \ref{Main3}  are contained in Section 
\ref{Proof2}. Although these results are connected we have stated them separately, 
since the argument for the $\ell^2$-singularity in $H^p$ 
also holds for $p = 1$. By contrast 
the proof of Theorem \ref{Main3} relies on properties of $h^p = L^p(\mathbb T, m)$ 
for $1 < p < \infty$, and it depends on the non-trivial fact  due to Dosev et al.  \cite{DJS} 
that the class $\mathcal S_{L^{p}}(L^p) \approx \mathcal  S_{L^{p}}(h^p) $ is additive.
Section \ref{Con} contains a number of further comments and open problems.
As an application of Section \ref{Proof2}  
we characterise the $\ell^2$-singular compositions 
$C_\phi \in \mathcal L(\VMOA)$.
As an additional motivation we also indicate a connection
 between a weaker version of Corollary \ref{main} and 
a general extrapolation result \cite{HST} for operators on $L^p$-spaces. 

A starting point for this paper was  a question by Jonathan Partington 
about the strict singularity of composition operators on $H^p$ for $p \neq 2$. 
We are indebted to 
Manuel Gonz\'alez, Francisco Hern\'andez and Dimitry Yakubovich for timely  questions 
towards Theorems \ref{Main2} and \ref{Main3}.

\section{Proof of Theorem \ref{Main1}}\label{Proof}

For $a \in \mathbb D$ and fixed $0 < p < \infty$ let  
\[
g_a(z) = \frac{(1 - \vert a\vert^2)^{1/p}}{(1- \overline{a}z)^{2/p}}, \quad z \in \mathbb D.
\]
Here $\Vert g_a\Vert_p = 1$, since for $p = 2$ 
the corresponding function  is the normalised reproducing kernel in $H^2$
associated to  $a \in \mathbb D$. The proof of Theorem \ref{Main1} is based 
on the following criterion: 
\textit{$C_\phi \in \mathcal K(H^p)$ if and only if}
\begin{equation}\label{Sh1}
\limsup_{\vert a \vert \to 1} \Vert C_\phi(g_a)\Vert_p = 0. 
\end{equation}
This is a restatement using the test functions $(g_a) \subset H^p$ of
a well-known characterisation of the compact operators $C_\phi \in \mathcal L(H^p)$
in terms of vanishing Carleson pull-back measures, see \cite[Thm. 3.12.(2)]{CM} 
(such a characterisation was first obtained by MacCluer \cite{MC} in the case of 
$H^p(B_N)$ for $N > 1$, where $B_N$ is the open euclidean ball in $\mathbb C^N$). 
Alternatively, \eqref{Sh1} is stated explicitly for $p = 2$ 
in e.g. \cite[5.4]{Sh1}, whereas 
the compactness of $C_\phi: H^p \to H^p$ is independent 
of $p \in (0,\infty)$ e.g. by \cite[Thm. 3.12.(2)]{CM}. 
After these preparations we proceed to the proof itself.

\begin{proof}[Proof of Theorem \ref{Main1}]
Suppose that $C_\phi \notin \mathcal K(H^p)$, where $1 \le p < \infty$. 
We will show by an explicit perturbation argument 
that  $C_\phi$ fixes a linearly isomorphic copy of $\ell^p$ in $H^p$.

Since condition \eqref{Sh1} fails there is $d > 0$ and a sequence $(a_n) \subset \mathbb D$ 
so that $\vert a_n\vert \to 1$ as $n \to \infty$ and
\begin{equation}\label{Sh2}
\Vert C_\phi(g_{a_{n}})\Vert_p \ge d > 0
\end{equation}
for all  $n \in \mathbb N$.
We may further assume  without loss of generality that 
$a_n \to 1$ as $n \to \infty$.
Namely, we may pass to a convergent subsequence 
in $\overline{\mathbb D}$ and compose $\phi$ with a 
suitable rotation of $\mathbb D$ that defines a linear isomorphism of $H^p$.  

Our starting point is the phenomenon that $(g_{a_{n}})$ admits subsequences 
which are small perturbations of a disjointly supported sequence in $L^p(\mathbb T, m)$, and 
hence  span an isomorphic copy of $\ell^p$. 
The crux of the argument is that this 
 can be achieved  simultaneously 
for further subsequences of  $(C_\phi(g_{a_{n}}))$, and the
following claim actually contains the basic step of the argument:

\begin{clm}\label{subsq} 
There is a subsequence of $(a_n)$, still denoted by $(a_n)$ for simplicity, 
for which there are constants  $c_1, c_2 > 0$ so that 
\begin{equation}\label{est}
c_1 \cdot \Vert (b_j)\Vert_{\ell^p} \le  \biggl\| \sum_{j=1}^\infty b_j C_\phi(g_{a_{j}})  \biggl\|_p 
 \le c_2 \cdot \Vert (b_j)\Vert_{\ell^p} \quad \textrm{for all}\  (b_j) \in \ell^p.
\end{equation}
\end{clm}

Assuming Claim \ref{subsq}  momentarily, the proof of Theorem \ref{Main1} is completed 
by using this claim a second time (formally in the case
 where $\phi(z) = z$ for $z \in \mathbb D$)
to extract a further subsequence of $(g_{a_{n}})$, still denoted 
by $(g_{a_{n}})$, so that 
\begin{equation}\label{est2}
d_1 \cdot \Vert (b_j)\Vert_{\ell^p} \le \biggl\| \sum_{j=1}^\infty b_j g_{a_{j}} \biggl\|_p 
 \le d_2 \cdot \Vert (b_j)\Vert_{\ell^p} \quad \textrm{for all}\  (b_j) \in \ell^p,
\end{equation}
for suitable constants $d_1, d_2 > 0$.
Then by combining \eqref{est} and \eqref{est2} we get  
\begin{align*}
\biggl\| \sum_{j=1}^\infty b_j C_\phi(g_{a_{j}}) \biggl\|_p & 
\ge c_1 \Vert (b_j) \Vert_p\\ 
& \ge c_1{d_2}^{-1} \biggl\| \sum_{j=1}^\infty b_j g_{a_{j}} \biggl\|_p,
\end{align*}
so that the restriction of $C_\phi$ defines a linear isomorphism $M \to C_\phi(M)$,
where $M = \overline{\operatorname{span}}\{g_{a_{j}}: j \in \mathbb N\} \approx \ell^p$.

Let  $A = \{\xi \in \mathbb T: \textrm{the radial limit}\ \phi(\xi)\  \textrm{exists}\}$
and 
\[
E_\varepsilon = \{\xi \in A: \vert \phi(\xi) - 1 \vert < \varepsilon\}
\]
for $\varepsilon > 0$. Recall that $\mathbb T \setminus A$ has measure zero.
The proof of Claim \ref{subsq} is an argument of gliding hump type based on
the following auxiliary observation. 

\begin{lemma}\label{aux1}
 Let $\phi$ and $(g_{a_{n}})$ be as above, where $a_n \to 1$ as $n \to \infty$.
 Then 
 \begin{itemize}
\item[(L1)] $\int_{\mathbb T \setminus E_{\varepsilon}} \vert C_\phi(g_{a_{n}})\vert^p dm \to 0$ as 
$n \to \infty$ for each fixed $\varepsilon > 0$,

\item[(L2)] $\int_{E_{\varepsilon}} \vert C_\phi(g_{a_{n}})\vert^p dm \to 0$ as $\varepsilon \to 0$ 
for each fixed $n \in \mathbb N$.
\end{itemize}
\end{lemma}

\begin{proof}
Observe first that 
\[
\int_{E_{\varepsilon}} \vert C_\phi(g)\vert^p dm \to 0
\]
as $\varepsilon \to 0$ for any $g \in H^p$, since $\cap_{\varepsilon > 0} E_{\varepsilon}
= \{\xi \in A: \phi(\xi) = 1\}$ has measure $0$ as $\phi$ is  non-constant. 
Moreover, if $\varepsilon > 0$ is fixed and
$\xi \in A \setminus E_{\varepsilon}$, then there is $n_{\varepsilon}$ such that
\[
\vert 1- \overline{a_n}\phi(\xi) \vert = \vert 1 - \phi(\xi) + \phi(\xi)(1 - \overline{a_n})\vert
\ge \vert 1 - \phi(\xi) \vert - \vert 1 - a_n\vert  >  \varepsilon/2 
\]
for all $n \ge n_{\varepsilon}$. It follows that 
\[
\vert C_\phi(g_{a_{n}})(\xi) \vert^p =  
\frac{1 - \vert a_n\vert^2}{\vert 1- \overline{a_n}\phi(\xi) \vert^2}
 \le 
\frac{4(1 - \vert a_n\vert^2)}{\varepsilon^2},
\]
so that (L1) holds as $n \to \infty$.
\end{proof}

To continue the argument of  Claim \ref{subsq} recall that 
$\int_{\mathbb T} \vert C_\phi(g_{a_{n}}) \vert^p dm \ge  d^p > 0$ by condition 
\eqref{Sh2}.  We  may then use  Lemma \ref{aux1} inductively 
to find indices  $j_1 < j_2 < \ldots$
and a decreasing sequence  $\varepsilon_j > \varepsilon_{j+1} \to 0$ so that 

\smallskip

(i) $\Big(\int_{E_{\varepsilon_{n}}} \vert C_\phi(g_{a_{j_{k}}})\vert^p dm\Big)^{1/p}  < 4^{-n}\delta d$ 
for all $k = 1,\ldots, n-1$,

(ii)  $\Big(\int_{\mathbb T \setminus E_{\varepsilon_{n}}} \vert C_\phi(g_{a_{j_{n}}})\vert^p dm 
\Big)^{1/p} < 4^{-n}\delta d$,

(iii) $\Big(\int_{E_{\varepsilon_{n}}} \vert C_\phi(g_{a_{j_{n}}})\vert^p dm \Big)^{1/p} > d/2$

\smallskip

\noindent for all $n \in \mathbb N$.   Here $\delta > 0$ is a small enough constant 
(to be chosen later). 
In fact, suppose that we have already found $a_{j_{1}}, \ldots , a_{j_{n-1}}$ and
$\varepsilon_1 > \ldots > \varepsilon_{n-1}$ satisfying (i) - (iii).
Then  property (L2) from  Lemma \ref{aux1} yields  
$\varepsilon_n < \varepsilon_{n-1}$ such that 
\[
\Big(\int_{E_{\varepsilon_{n}}} \vert C_\phi(g_{a_{j_{k}}})\vert^p dm \Big)^{1/p} <  4^{-n}\delta d
\]
for each $k = 1, \ldots , n-1$. After this use property (L1) from  Lemma \ref{aux1} together with 
\eqref{Sh2}  to find an index $j_n > j_{n-1}$ so that conditions (ii) and (iii) are satisfied for
the set $E_{\varepsilon_{n}}$.

In the interest of notational simplicity we  relabel $a_{j_{n}}$ as $a_{n}$ for 
$n \in \mathbb N$.  The idea of the argument is that the sequence
$(C_\phi(g_{a_{n}}))$ essentially 
resemble disjointly supported peaks  in $L^p(\mathbb T, m)$ close to the point $1$.  
We will next verify the left-hand inequality in \eqref{est} 
by a direct  perturbation argument.
Let  $b = (b_j) \in \ell^p$ be arbitrary.  Our starting point will be the identity
 \begin{equation}\label{sp}
\biggl\|  \sum_{j=1}^\infty b_j C_\phi(g_{a_{j}}) \biggl\|_p^p = 
 \sum_{n=0}^\infty \int_{E_{\varepsilon_{n}} \setminus E_{\varepsilon_{n+1}}} 
 \biggl|  \sum_{j=1}^\infty b_j C_\phi(g_{a_{j}})  \biggl|p dm,
 \end{equation}
 where we set $E_{\varepsilon_{0}} = \mathbb T$.
 
 Observe first that for each $n \in \mathbb N$ we get that
 \begin{align*}\label{est3}
 \Bigg(\int_{E_{\varepsilon_{n}} \setminus E_{\varepsilon_{n+1}}} 
 \vert C_\phi(g_{a_{n}})\vert^p dm\Bigg)^{1/p} & =  \Bigg(\int_{E_{\varepsilon_{n}}} 
 \vert C_\phi(g_{a_{n}})\vert^p dm -  \int_{E_{\varepsilon_{n+1}}} 
 \vert C_\phi(g_{a_{n}})\vert^p dm\Bigg)^{1/p}\\
 &  > \bigg( (\frac{d}{2})^p - (4^{-n-1}\delta d)^p\bigg)^{1/p} \ge
 \frac{d}{2} - 4^{-n-1}\delta d
 \end{align*}
 in view of (i) and (iii), where the last estimate holds because $0 < 1/p \le 1$.
Moreover, note that  
 \begin{equation*}\label{est4}
 \Bigg( \int_{E_{\varepsilon_{n}} \setminus E_{\varepsilon_{n+1}}} 
 \vert C_\phi(g_{a_{j}})\vert^p dm \Bigg)^{1/p} < 2^{-n-j}\delta d
 \end{equation*} 
  for all $j \neq n$. In fact, 
  $\Big( \int_{E_{\varepsilon_{n}} \setminus E_{\varepsilon_{n+1}}} 
 \vert C_\phi(g_{a_{j}})\vert^p dm \Big)^{1/p}$ is dominated by $4^{-n}\delta d$ for 
$j < n$ and by $4^{-j}\delta d$ for $j > n$ in view of (i) and (ii). 

Thus we get from the triangle inequality in $L^p$, 
together with the preceding estimates,
that for all $n \in \mathbb N$
one has
\begin{align*}
\Big( & \int_{E_{\varepsilon_{n}} \setminus E_{\varepsilon_{n+1}}}  \vert \sum_{j=1}^\infty b_j 
 C_\phi(g_{a_{j}}) \vert^p  dm \Big)^{1/p}  \\
 &  \ge  
 \vert b_n\vert \Bigg( \int_{E_{\varepsilon_{n}} \setminus E_{\varepsilon_{n+1}}} 
 \vert C_\phi(g_{a_{n}})\vert^p dm \Bigg)^{1/p} 
  \quad -
 \sum_{j \neq n}  \vert b_j\vert \Bigg( \int_{E_{\varepsilon_{n}} \setminus E_{\varepsilon_{n+1}}} 
 \vert C_\phi(g_{a_{j}})\vert^p dm \Bigg)^{1/p} \\
 & \ge   \vert b_n\vert (\frac{d}{2} - 4^{-n-1}d\delta) - 2^{-n} \delta d \Vert b\Vert_p
 \ge \frac{d}{2} \vert b_n\vert -  2^{-n+1} \delta d \Vert b\Vert_p. 
\end{align*}
By summing over $n$ we get from the disjointness and  the triangle inequality in $\ell^p$ that
\begin{align*}
\biggl\|  \sum_{j=1}^\infty b_j C_\phi(g_{a_{j}}) \biggl\|_p  & \ge 
\Big( \sum_{n=1}^\infty   \biggl|\ \frac{d}{2} \vert b_n\vert - 
2^{-n+1} \delta d  \Vert b\Vert_p\  \biggl| ^p \Big)^{1/p} \\
 & \ge \frac{d}{2} \Big(\sum_{n=1}^\infty \vert b_n\vert^p\Big)^{1/p} - 
 \delta d  \Vert b\Vert_p \Big(\sum_{n=1}^{\infty} 2^{-(n-1)p}\big)^{1/p} \\
 & \ge d \Big(\frac{1}{2} - \delta \cdot (1 - 2^{-p})^{1/p} \Big) \Vert b\Vert_p
\ge \frac{d}{4}\Vert b\Vert_p,  
\end{align*}
where the last estimate holds once we choose $\delta > 0$ small enough,
 so that  $\delta \cdot (1 - 2^{-p})^{1/p} < 1/4$.

The proof of the right-hand inequality in \eqref{est} is a straightforward variant of the
preceding estimates. This inequality
does not affect the choice of $\delta > 0$, 
and hence the details will be omitted here.
\end{proof}

\begin{remarks}\label{qnorm}
The  definitions of the classes $\mathcal K(H^p)$, $\mathcal S(H^p)$ and 
$\mathcal S_p(H^p)$ also make sense in the range $0 < p < 1$, where  $H^p$ are 
only quasi-Banach spaces.
The composition operators $C_\phi$ are continuous on $H^p$
for $0 < p < 1$, and Theorem \ref{Main1} 
 as well as Corollary \ref{main} remain true here. The argument is similar
to the above, but the quasi-norms $\Vert \cdot\Vert_p$
in $H^p$ as well as in $\ell^p$ are only $p$-norms, that is,
$\Vert f+g\Vert_p^p \le \Vert f\Vert_p^p + \Vert g\Vert_p^p$ for $f, g \in H^p$. 
This will affect a few constants when applying the triangle inequalities 
as in the proof of Theorem \ref{Main1}. However, 
we leave the precise details for $0 < p < 1$ to the interested reader, since  
a full analogue of Theorem \ref{Main0} appears out of reach. In fact,
$H^p$ and $L^p$ are not isomorphic for $0 < p < 1$ 
(see e.g. \cite[p. 35]{KPR}), the structure of their respective subspaces differs
(see e.g. \cite[chapter 3.2]{KPR}), and 
no version of  \eqref{Weis} appears known.
\end{remarks}

\begin{remarks}\label{comp}
It is possible to ensure in the proof of Theorem \ref{Main1} that the subspaces 
$M = \overline{\operatorname{span}}\{g_{a_{j}}: j \in \mathbb N\}$ 
and $C_\phi(M)$ are both complemented in $H^p$ 
(this also follows from the general result in \cite{W} for $1 < p < \infty$).
For this one uses the fact that the closed linear span of 
a disjointly supported sequence
is complemented in $L^p(\mathbb T, m)$,
a classical perturbation argument (cf. \cite[1.a.9]{LT}), 
as well as the complementation of  $H^p \subset L^p(\mathbb T, m)$
for $1 < p < \infty$.
Note also that for $p = 1$ the argument of  Theorem \ref{Main1} 
provides an alternative route
to the weak compactness characterisation of Sarason \cite{S1} 
cited in Section \ref{intro}. In fact,  if 
$C_\phi \notin \mathcal K(H^1)$, then $C_\phi$ fixes a copy of the non-reflexive space
$\ell^1$ by Theorem \ref{Main1}, whence
$C_\phi$ is not a weakly compact operator  $H^1 \to H^1$.  
\end{remarks}

\section{Proof of Theorems \ref{Main2} and \ref{Main3}}\label{Proof2}

The proof of Theorem \ref{Main2} is contained in the following three results.
We first look separately at the case $p = 2$.
Recall our notation  $E_\phi = \{ e^{i\theta} : |\phi(e^{i\theta})| = 1 \}$ for
analytic maps $\phi: \mathbb D \to \mathbb D$.

\begin{lemma} \label{le:phipowers}
Suppose that condition \eqref{CM} fails, that is, $m(E_\phi) > 0$.
Then there exist integers $0 \leq n_1 < n_2 < \cdots$ and a constant
$K > 0$ such that
\[
   K^{-1} \cdot \|c\|_{\ell^2}
   \leq \biggl\| \sum_{k=1}^\infty c_k\phi^{n_k} \biggr\|_{2}
   \leq K \cdot \|c\|_{\ell^2}
\]
for all $c = (c_k) \in \ell^2$.
\end{lemma}

\begin{proof}
The upper estimate follows from the boundedness of $C_\phi$ on $H^2$
and the orthonormality of the sequence  $(z^n)$ in $H^2$.

To establish the lower estimate, note that $z^n \to 0$ weakly and
therefore also $\phi^n = C_\phi(z^n) \to 0$ weakly in $H^2$ as $n\to\infty$.
Hence we may set $n_1 = 0$ and then proceed inductively to pick increasing indices $n_k$
such that the inner-products satisfy 
$|(\phi^{n_j},\phi^{n_k})| \leq 2^{-2k} m(E_\phi)$ for all
$1\leq j < k$ and each $k \in \mathbb N$. Let $c=(c_k) \in \ell^2$ be arbitrary
and note that
\[
   \biggl\| \sum_{k=1}^\infty c_k\phi^{n_k} \biggr\|_{2}^2
   = \sum_{k=1}^\infty |c_k|^2\|\phi^{n_k}\|_{2}^2 +
     2 \Re \sum_{k=1}^\infty\sum_{j=1}^{k-1} c_j \bar{c_k}
           (\phi^{n_j},\phi^{n_k}).
\]
Obviously $\|\phi^{n_k}\|_{2}^2 \geq \int_{E_{\phi} } \vert  \phi^{n_k} \vert^2 dm =  
m(E_\phi)$ for each $k$.
Moreover, we get that
\[ \begin{split}
   \biggl| \sum_{k=1}^\infty & \sum_{j=1}^{k-1} c_j \bar{c_k}
     (\phi^{n_j},\phi^{n_k}) \biggr|
   \leq \|c\|_{\ell^2}^2 \sum_{k=1}^\infty\sum_{j=1}^{k-1}
     2^{-2k} m(E_\phi)  \\
   &\leq \tfrac{1}{2} \|c\|_{\ell^2}^2 m(E_\phi)
     \sum_{k=1}^\infty\sum_{j=1}^{k-1} 2^{-k}2^{-j}  
   = \tfrac{1}{6} \|c\|_{\ell^2}^2 m(E_\phi).
\end{split} \]
By combining these estimates we obtain the desired lower bound 
\[
   \biggl\| \sum_{k=1}^\infty c_k\phi^{n_k} \biggr\|_{2}^2
   \geq \|c\|_{\ell^2}^2 m(E_\phi)
        - \tfrac{1}{2}\|c\|_{\ell^2}^2 m(E_\phi)
   = (\tfrac{1}{2} m(E_\phi)) \|c\|_{\ell^2}^2 .
\]
\end{proof}

In order to treat general $p \in [1,\infty)$ recall that the analytic map 
$f: \mathbb D \to \mathbb C$  belongs to  $\BMOA$ if 
\[
\vert f \vert_{*}= \sup_{a \in \mathbb D} \Vert f \circ \sigma_a - f(a)\Vert_2 < \infty,
\]
where $\sigma_a(z) = \frac{a-z}{1 - \overline{a}z}$ is the Möbius-automorphism of  $\UnitDisk$
interchanging $0$ and $a$ for $a \in \UnitDisk$. The Banach space $\BMOA$ is normed by
$\Vert f \Vert_{\BMOA} = \vert f(0) \vert + \vert f\vert_{*}$.
Moreover, $\VMOA$ is the closed subspace of $\BMOA$, where $f \in \VMOA$ if
\[
\lim_{\vert a\vert \to 1} \Vert f \circ \sigma_a - f(a)\Vert_2 = 0.
\]
We refer to e.g. \cite{Ga} and \cite{G} for background on $\BMOA$.
It follows readily from Littlewood's subordination 
 theorem that  $C_\phi$ is bounded $\BMOA \to \BMOA$
for any analytic map $\phi: \mathbb D \to \mathbb D$, see e.g. \cite[p. 2184]{BCM}. 

The following proposition establishes one implication of Theorem \ref{Main2}.

\begin{proposition}\label{s2}
Let $1 \leq p < \infty$ and suppose that  $m(E_\phi) > 0$.
Then there exist increasing integers
$0 \le n_1 < n_2 < \cdots$ such that the subspace
\[
M = \overline{\operatorname{span}}\{z^{n_k}: k \geq 1\} \subset H^p
\]
is isomorphic to $\ell^2$ and the restriction ${C_\phi}_{|M}$ is bounded below on $M$.
Hence $C_\phi \notin \mathcal S_2(H^p)$.
\end{proposition}

\begin{proof}
We start by choosing the increasing integers $(n_k)$ as in Lemma~\ref{le:phipowers}.
By passing to a subsequence we may also assume that $(z^{n_k})$ is a
lacunary sequence, that is, \ $\inf_k (n_{k+1}/n_k) > 1$. Paley's theorem 
(see e.g. \cite[p. 104]{D}) implies that for $1 \le p < \infty$ the sequence 
$(z^{n_k})$ is equivalent in $H^p$ to
the unit vector basis basis of $\ell^2$, that is,
\begin{equation}\label{Pa}
\bigg\| \sum_{k=1}^\infty c_k z^{n_k} \bigg\|_{p} \sim \|c\|_{\ell^2}
\end{equation}
for all $c = (c_k) \in \ell^2$. 
(Here, and in the sequel, we use $\sim$ as a short-hand 
notation for the equivalence of the 
respective norms.)

\emph{Case $p \geq 2$.}
By Hölder's inequality
and Lemma~\ref{le:phipowers} we have that
\[
   \biggl\| C_\phi\biggl(\sum_{k=1}^\infty c_k z^{n_k} \biggr)\biggr\|_{p}
   =    \biggl\| \sum_{k=1}^\infty c_k \phi^{n_k} \biggr\|_{p}
   \geq \biggl\| \sum_{k=1}^\infty c_k \phi^{n_k} \biggr\|_{2}
   \sim \|c\|_{\ell^2}.
\]
According to  \eqref{Pa} and the boundedness of $C_\phi$ 
this proves the claim for $p \geq 2$.

\smallskip

\emph{Case $1 \leq p < 2$.}
We start by invoking a version of Paley's theorem for $\mathit{BMOA}$
(see e.g. \cite[Sec.~9]{G}), which together with the boundedness of
$C_\phi$ on $\mathit{BMOA}$ ensures that
\[
   \biggl\| \sum_{k=1}^\infty c_k \phi^{n_k} \biggr\|_{\mathit{BMOA}}
   \leq \Vert C_\phi\Vert \cdot  \biggl\| \sum_{k=1}^\infty c_k z^{n_k} \biggr\|_{\mathit{BMOA}}
   \leq K \cdot \Vert C_\phi\Vert \cdot \|c\|_{\ell^2}
\]
for all $c = (c_k) \in \ell^2$ and a uniform constant $K > 0$. In view of
Fefferman's $H^1$-$\mathit{BMOA}$ duality pairing (see e.g. \cite[Sec.~7]{G}) 
we may further estimate
\[ \begin{split}
   \biggl\| \sum_{k=1}^\infty c_k \phi^{n_k} \biggr\|_{\mathit{BMOA}}
   \biggl\| \sum_{k=1}^\infty c_k \phi^{n_k} \biggr\|_{1}
   &\geq  \biggl| \big( \sum_{k=1}^\infty c_k \phi^{n_k},
           \sum_{k=1}^\infty c_k \phi^{n_k} \big) \biggr|  \\
   &=  \biggl\| \sum_{k=1}^\infty c_k \phi^{n_k} \biggr\|_{2}^2
   \sim \|c\|_{\ell^2}^2,
\end{split} \]
where we again use Lemma~\ref{le:phipowers} at the final step.
By applying Hölder's inequality and combining the preceding estimates we obtain that
\[
   \biggl\| \sum_{k=1}^\infty c_k \phi^{n_k} \biggr\|_{p}
   \geq \biggl\| \sum_{k=1}^\infty c_k \phi^{n_k} \biggr\|_{1}
   \geq K' \|c\|_{\ell^2}
\]
for some uniform constant $K' > 0$. In particular, $C_\phi \notin \mathcal S_2(H^p)$
in view of \eqref{Pa},  which completes the verification of the proposition  for $1 \leq p < 2$.
\end{proof}

The converse implication in Theorem \ref{Main2} is contained in the following 

\begin{proposition}\label{ell2}
Let $1 \leq p < \infty$, $p \neq 2$, and suppose that $m(E_\phi) = 0$.
 If $(f_n)$ is any normalized sequence in $H^p$ which is equivalent to the unit vector basis
of $\ell^2$, then $C_\phi$ is not bounded  below on
$\overline{\operatorname{span}}\{f_n: n \in \mathbb N\} \subset H^p$. In particular,
$C_\phi \in \mathcal S_2(H^p)$. 
\end{proposition}

\begin{proof}
Assume to the contrary that
\begin{equation}\label{eq:ell2}
   \biggl\| \sum_{n=1}^\infty c_n C_\phi(f_n) \biggr\|_{p}
   \sim \biggl\| \sum_{n=1}^\infty c_n f_n \biggr\|_{p}
   \sim \Vert c \Vert_{\ell^2}^2
\end{equation}
for all sequences $c = (c_n) \in \ell^2$. In particular, 
$\| C_\phi(f_n) \|_{p} \ge d > 0$ for all $n$ and some constant $d$.
We write $E_k = \{e^{i\theta}: |\phi(e^{i\theta})| \geq 1-\tfrac{1}{k}\}$ 
for $k \geq 1$. Since
$\lim_{k\to\infty} m(E_k) = m(E_\phi) = 0$, we get that
\[
   \lim_{k\to\infty} \int_{E_k} | C_\phi(f_n) |^p \,dm = 0
   \quad\text{for all $n$.}
\]
On the other hand, $f_n \to 0$ weakly in $H^p$ 
and hence $f_n \to 0$
uniformly on compact subsets of $\UnitDisk$ as $n \to \infty$. 
This implies that
\[
   \lim_{n\to\infty} \int_{\T\setminus E_k} | C_\phi(f_n) |^p \,dm = 0
   \quad\text{for all $k$.}
\]
By using the above properties and proceeding recursively in a similar fashion
to the  argument for Theorem \ref{Main1} in section \ref{Proof} we find
increasing sequences of integers $0 \le n_1 < n_2 < \cdots$ and $1 = k_1 < k_2 < \cdots$,
such that 
\[
   \biggl\| \sum_{j=1}^\infty c_j  C_\phi(f_{n_{j}}) \biggr\|_{p}^p
   = \sum_{l=1}^\infty \int_{E_{k_l}\setminus E_{k_{l+1}}}
     \biggl| \sum_{j=1}^\infty c_j C_\phi(f_{n_{j}}) \biggr|^p \,dm
   \sim \Vert c\Vert_{\ell^p}^p
\]
holds for all $c = (c_j) \in \ell^p$ with uniform constants.
However, for $p \neq 2$ such estimates obviously contradict \eqref{eq:ell2}. Thus
$C_\phi \in S_2(H^p)$, and this completes the proof
of the Proposition (and hence also of Theorem \ref{Main2}).
\end{proof}

We remind that Theorem \ref{Main2} does not hold for $p = 2$. 
The result easily yields very explicit examples of  operators
$C_\phi \in \mathcal S_2(H^p) \setminus \mathcal S_p(H^p)$.

\begin{example}\label{Ex}
Let $\phi(z) =  \frac{1}{2}(1+z)$ for $z \in \mathbb D$. Theorem \ref{Main2} 
implies that $C_\phi$ does not fix any copies of $\ell^2$ in $H^p$. 
On the other hand, it is well known 
that  $C_\phi \notin \mathcal K(H^p)$,
see e.g. \cite[Sec. 2.5]{Sh2}, so that $C_\phi$ does fix copies of $\ell^p$ in $H^p$
by Theorem \ref{Main1}.
\end{example}

We next prepare for the proof of Theorem \ref{Main3}. 
This involves the harmonic Hardy space $h^p$, that is, 
the space of complex-valued harmonic functions 
$f:  \UnitDisk \to \mathbb C$ normed by \eqref{Hnorm}.  Recall that for $1 < p < \infty$ 
there is a well-known isometric identification  $h^p = L^p(\mathbb T,m)$ as a
complex Banach space.
Here $f \in h^p$ corresponds to its a.e. radial limit function $f \in  L^p(\mathbb T,m)$,
whereas conversely $g \in  L^p(\mathbb T,m)$ determines the harmonic extension 
$P[g] \in h^p$ through the Poisson integral. Moreover,  
$h^p = H^p \oplus \overline{H^p_0}$,
where $H^p_0 = \{f \in H^p: f(0) = 0\}$ and $\overline{H^p_0} = \{ \overline{f}: f \in H^p_0\}$.

Let $\phi: \mathbb D \to \mathbb D$ be any analytic map. 
The Littlewood subordination theorem for subharmonic functions 
(see e.g. \cite[Thm. 2.22]{CM}) implies 
that the composition operator $f \mapsto f \circ \phi$ is also bounded $h^p \to h^p$
for $1 \le p < \infty$. 
It will be convenient in the argument to use the notation
$\widetilde{C_\phi}(f) = f \circ \phi$ for $f \in h^p$
 to distinguish the composition operator 
on $h^p$ from its relative  on $H^p$.
In particular, if in addition $\phi(0) = 0$, then we may decompose  
\begin{equation}\label{split}
\widetilde{C_\phi} =
\left( \begin{array}{ccc}
C_\phi & 0\\
0 & C_\phi\\
\end{array} \right), \quad \widetilde{C_\phi}(f,g) = (f \circ \phi, g \circ \phi),
\end{equation}
as a matrix direct sum   with respect to the decomposition 
$h^p = H^p \oplus \overline{H^p_0}$. Here  $\phi(0) = 0$ ensures 
that $g \circ \phi \in H^p_0$ for any $g \in H^p_0$.

\begin{proof}[Proof of Theorem \ref{Main3}] 
We may assume during the proof that $\phi(0)=0$. In fact, otherwise 
consider  $\psi = \sigma_{\phi(0)} \circ \phi$, where 
$\sigma_{\phi(0)}: \mathbb D \to \mathbb D$ is the automorphism 
interchanging $0$ and $\phi(0)$.
Then $\psi(0) = 0$ and
$\widetilde{C_{\psi}}  = \widetilde{C_\phi} \circ \widetilde{C}_{\sigma_{\phi(0)}}$, 
where $\widetilde{C}_{\sigma_{\phi(0)}}$ is a linear isomorphism $h^p \to h^p$ 
(as well as $H^p \to H^p$),
which does not affect any of the  claims of the theorem. 
 
The proof of the implication (iii) $\Rightarrow$ (i) is contained in the following claim.

\begin{clm}\label{harm} 
Let $1 < p < \infty$ and suppose that $m(E_\phi) > 0$.
Then $\widetilde{C_\phi} \notin \mathcal S_{L^p}(h^p)$, that is,  there is  a subspace
$M \subset h^p$, $M \approx L^p$, such that
$\widetilde{C_\phi}_{|M}$ is bounded  below.
\end{clm}

To prove the claim define the Borel measure $\nu$ on $\T$
by $\nu(A) = m((\phi)^{-1}(A))$. Then
$\nu$ is absolutely continuous:  if $A \subset \T$ is a Borel set
and $u_A = P[\chi_A]$ is the harmonic extension (i.e.\ the Poisson integral) of
$\chi_A$, we have that
\[
   \nu(A) = \int_{(\phi^*)^{-1}(A)} dm
      \leq \int_\T u_A \circ\phi \,dm
      = u_A(\phi(0)) = u_A(0) = m(A).
\]
Since $\nu(\T) = m(E_\phi) > 0$, it follows that the density  
$d\nu/dm \geq \delta$ for some $\delta > 0$ on a Borel set $F \subset \T$ of positive 
Lebesgue measure.

We may now choose $M = L^p(F,m)$. Indeed, given any $f \in L^p(F,m)$,
we have
\[
   \|\widetilde{C_\phi} f\|_{L^p}^p \geq \int_{E_\phi} |f\circ\phi|^p \,dm
   = \int_\T |f|^p \,d\nu \geq \delta \int_F |f|^p \,dm
   = \delta \|f\|_{L^p(F,m)}^p,
\]
which establishes Claim \ref{harm}, since $L^p(F,m) \approx L^p$.

The implication (ii) $\Rightarrow$ (iii) follows from \eqref{split} and the non-trivial result that
the class $\mathcal S_{L^p}(L^p) \approx \mathcal S_{L^p}(h^p)$ is additive, 
see \cite[p. 103 and 105]{DJS}.
In fact, if $C_\phi \in \mathcal S_{L^p}(H^p)$, then 
\[
\widetilde{C_\phi} = \left( \begin{array}{ccc}
C_\phi & 0\\
0 & 0\\
\end{array} \right) +
\left( \begin{array}{ccc}
0 & 0\\
0 & C_\phi\\
\end{array} \right)
\]
is the sum of two operators from  $\mathcal S_{L^p}(h^p)$, and hence $L^p$-singular
by additivity.

Finally, the proof of the implication  (i) $\Rightarrow$ (ii) is already contained in 
that of Proposition \ref{ell2}. In fact, 
if there is a subspace $M \subset H^p$, $M \approx L^p$, so that $C_\phi$ is an isomorphism 
$M \to C_\phi(M)$, then $C_\phi$ also fixes 
the isomorphic copies of $\ell^2$ contained in $M$. It was shown in Proposition \ref{ell2}
that the latter property is incompatible with \eqref{CM}. 
\end{proof}

We note that Claim \ref{harm} also holds for $p = 1$. However, 
there is no immediate analogue of Theorem \ref{Main3} for $H^1$. In fact, 
$\mathcal S_{L^1}(H^1) = \mathcal L(H^1)$, since $L^1$
does not embed isomorphically into $H^1$, see e.g. \cite[1.d.1]{LT}.

In conclusion, recall that there are infinitely many norm-closed  ideals 
$\mathcal I$ of $\mathcal L(H^p)$ 
satisfying  $\mathcal S_2(H^p) \subset \mathcal I \subset \mathcal S_{L^p}(H^p)$ 
for $1 < p < \infty$ and $p \neq 2$, see \cite[5.3.9]{P}.
By contrast,  Theorems \ref{Main2} and \ref{Main3} imply that
there is no corresponding gradation for composition operators on $H^p$. 
In some cases the trichotomy of Theorem \ref{Main0} can be sharpened by combining with known results about the subspaces of $H^p \approx L^p$. For instance, for $2 < p < \infty$ 
it follows from a result of Johnson and Odell \cite[Thm. 1]{JO} that if ${C_\phi}_{|M}$ is bounded below on an infinite-dimensional subspace $M \subset H^p$ that contains no 
isomorphic copies of $\ell^2$,
then $M$ embeds isomorphically into $\ell^p$, whence $C_\phi \notin \mathcal S_p(H^p)$.

\section{Concluding remarks and questions}\label{Con}

In this section we list some further examples of Banach spaces of analytic functions, where composition operators have related rigidity properties, and draw attention to open problems. We also sketch another approach towards Theorem \ref{Main1} which motivated this paper, though 
its conclusion is much weaker.

\subsection{Further rigidity properties.}\label{rigidity}

The weaker rigidity property 
\begin{equation}\label{rig}
C_\phi \in \mathcal S(E) \ \textrm{if and only if}\ C_\phi \in \mathcal K(E)
\end{equation}
holds for many other Banach spaces $E$ of 
analytic functions on $\mathbb D$ apart from the Hardy spaces.
The following list briefly recalls some cases. Typically these results were not stated
in terms of strict singularity, and as a rule they do not yield as precise information
as our results for $H^p$.

\smallskip

$\bullet$ 
The following dichotomy in \cite[Thm. 1]{BDL} is 
an explicit precursor of  Theorem \ref{Main1}:
\textit{either $C_\phi \in \mathcal K(H^\infty_v)$ or  $C_\phi \notin \mathcal S_\infty(H^\infty_v)$}.
Here $H^\infty_v$ is the weighted $H^\infty$-space  for a strictly positive 
weight function $v$ on  $\mathbb D$.  It is also possible to deduce 
versions of \eqref{rig} for $H^\infty$ (the case $v \equiv 1$)  
from even earlier results.
In fact, it follows from any of the references \cite{Ul}, \cite{ADL} or \cite{CD} that 
$C_\phi \in \mathcal L(H^\infty)$ is weakly compact 
if and only if $C_\phi \in \mathcal K(H^\infty)$. Moreover, 
 Bourgain \cite{B} established that 
$\mathcal W(H^\infty, X) = \mathcal S_\infty(H^\infty, X)$ for any Banach space $X$,
where $\mathcal W$ denotes the class of weakly compact operators.
Here $\mathcal K(H^\infty) \varsubsetneq \mathcal S(H^\infty)$, since this 
holds for the complemented subspace $\ell^\infty$ of $H^\infty$.

$\bullet$ 
The dichotomy in  Theorem \ref{Main1} holds  for arbitrary bounded operators on 
 the Bergman space $A^p$. In fact, $A^p \approx \ell^p$ for $1 \le p < \infty$ 
by a result of Lindenstrauss and Pe{\l}cynski, see \cite[Thm. III.A.11]{Wo}, whereas 
$\mathcal S(\ell^p) = \mathcal S_p(\ell^p) = \mathcal K(\ell^p)$ 
by a result of  Gohberg, Markus and Feldman,
see \cite[5.1-5.2]{P}.

$\bullet$ 
It is known that  the Bloch space $\mathcal B$ is isomorphic to $\ell^\infty$, while
$C_\phi \in \mathcal W(\mathcal B)$ if and only if $C_\phi \in \mathcal  K(\mathcal B)$, 
see e.g. \cite[Cor. 5]{LST}. Moreover, any 
$U \notin \mathcal W(\ell^\infty, X)$ fixes a copy of  $\ell^\infty$ 
for any Banach space $X$,  see \cite[2.f.4]{LT}. 
Consequently either $C_\phi \in \mathcal  K(\mathcal B)$ or 
$C_\phi \notin \mathcal S_\infty(\mathcal B)$.

$\bullet$ It follows from \cite[section 3]{LNST} that
$C_\phi \in \mathcal K(\BMOA)$ if and only of $C_\phi \in \mathcal S_{c_{0}}(\BMOA)$. 
In fact, the argument shows that if  $C_\phi \notin \mathcal K(\BMOA)$, 
then there is $M \subset \VMOA$, $M \approx c_0$, so that ${C_\phi}_{|M}$
is bounded below.  Here again
$\mathcal K(BMOA) \varsubsetneq   \mathcal S_{c_{0}}(\BMOA)$, since $\BMOA$ 
contains complemented subspaces isomorphic to $\ell^2$ in view of Paley's theorem
(see e.g. \cite[Thm. 9.2]{G}). 

\smallskip

Actually the results of section \ref{Proof2} combined with \cite{LNST} 
lead to a better understanding of the
$\ell^2$-singular composition operators on $\VMOA$ and $\BMOA$. 

\begin{proposition}\label{BMOA}
(i) If $\phi: \mathbb D \to \mathbb D$ is an analytic map, and 
$C_\phi \in \mathcal S_2(\BMOA)$, then \eqref{CM} holds (that is, $m(E_\phi) = 0$).

(ii) If $\phi \in \VMOA$, then $C_\phi \in \mathcal S_2(\VMOA)$ if (and only if) 
\eqref{CM} holds.
\end{proposition}

\begin{proof}
(i)  The argument  is essentially contained in that
 of Proposition \ref{s2}.
In fact, suppose that $m(E_\phi) > 0$, 
where $E_\phi = \{ e^{i\theta} : |\phi(e^{i\theta})| = 1 \}$.
Then the proof of the case $1 \le p < 2$ of Proposition \ref{s2}
gives  a lacunary sequence $(n_k)$ and constants $K_1, K_2 > 0$  so that in the 
$H^1$-$\mathit{BMOA}$ duality pairing
\[
   \biggl\| \sum_{k=1}^\infty c_k \phi^{n_k} \biggr\|_{\mathit{BMOA}}
   \biggl\| \sum_{k=1}^\infty  c_k \phi^{n_k} \biggr\|_{1} \geq K_1 \cdot \|c\|_{\ell^2}^2,
   \]
  as well as $\| \sum_{k=1}^\infty  c_k \phi^{n_k} \|_{1} \leq K_2 \cdot \|c\|_{\ell^2}$ for all 
  $c = (c_k) \in \ell^2$. Since $C_\phi$ is bounded on $\BMOA$ 
  it follows as before from Paley's theorem  in $\BMOA$
that $C_\phi$ is bounded below on 
  $\overline{\operatorname{span}}\{z^{n_{k}}: k \in \mathbb N\} \approx \ell^2$ in $\BMOA$.
  
  (ii) Recall that $C_\phi: \VMOA \to \VMOA$ if $\phi \in \VMOA$, see e.g. \cite[Prop. 2.3]{BCM}.
  Assume that $m(E_\phi) = 0$
  and suppose to the contrary that there is a normalised sequence $(f_k) \subset \VMOA$ 
equivalent to the unit vector basis of $\ell^2$, for which 
  \begin{equation}\label{VMOA}
  \biggl\| \sum_{k=1}^\infty  c_k C_\phi(f_k)  \biggl\|_{\BMOA} \sim \Vert c\Vert_{\ell^2}
  \end{equation}
  for all $ c = (c_k) \in \ell^2$. In particular, 
  $\Vert f_k \circ \phi \Vert_{\BMOA} \ge d > 0$ for all $k$, while $(f_k)$ is 
  weak-null sequence in $\VMOA$, so that $f_k \to 0$ uniformly on compact subsets of 
  $\mathbb D$ as $k \to \infty$. Moreover, by the John-Nirenberg inequality there
  is a uniform constant $c > 0$ so that
  \[
  \Vert f_k \circ \phi \Vert_{4} \le c \Vert f_k \circ \phi \Vert_{\BMOA}, \quad k \in \mathbb N.
  \]
 
 Let $E_k = \{ e^{i\theta} : |\phi(e^{i\theta})| \ge  1 - \frac{1}{k}\}$ for $k \in \mathbb N$. 
 From the above estimates and Hölder's inequality we get that
 \begin{align*}
  \Vert f_n \circ \phi \Vert_{2}^2 & = \int_{E_{k}}   \vert f_n \circ \phi \vert^2 dm + 
  \int_{\mathbb T \setminus E_{k}}   \vert f_n \circ \phi \vert^2 dm\\
  & \le \big(\int_{E_{k}}   \vert f_n \circ \phi \vert^4 dm\big)^{1/2} \sqrt{m(E_k)} 
  +   \int_{\mathbb T \setminus E_{k}}   \vert f_n \circ \phi \vert^2 dm. 
 \end{align*} 
 Since $\int_{\mathbb T \setminus E_{k}}   \vert f_n \circ \phi \vert^2 dm \to 0$ for each $k$
 as $n \to \infty$, we obtain that 
 \[
 \limsup_{n\to\infty}   \Vert f_n \circ \phi \Vert_{2}^2  \le C \sqrt{m(E_k)}
 \]
 for some constant $C > 0$ independent of $k \in \mathbb N$.  
 By letting $k \to \infty$ and using that $m(E_\phi) = 0$
 we deduce that  $\lim_{n \to \infty}  \Vert f_n \circ \phi \Vert_{2} = 0$.
 
By \cite[Prop. 6]{LNST}  there is  a subsequence $(f_{n_{k}} \circ \phi)$  such that
 \[
\Vert \sum_{k=1}^\infty c_k f_{n_{k}} \circ \phi \Vert_{\BMOA} \sim \Vert c \Vert_{\ell^\infty}
 \]
holds  for all $c = (c_k) \in c_0$. Obviously this contradicts \eqref{VMOA}.
\end{proof}

\subsection{An alternative approach}\label{alt}

We next indicate a different approach  towards a weaker version of Theorem \ref{Main1},
 which highlights a connection to the following general  interpolation-extrapolation 
theorem for strictly singular operators on $L^p$-spaces  due to
Hernandez et.al. \cite[Thm. 3.8]{HST}:  
\textit{Let $1 \le p < q \le \infty$, and 
assume that the linear operator $T$ is bounded 
$L^p \to L^p$ and $L^q \to L^q$. Moreover,
suppose further that
there is $r \in (p,q)$ for which $T \in \mathcal S(L^r)$.
Then $T \in \mathcal K(L^s)$ for all $p < s < q$.} 

To apply the above result suppose that $C_\phi \in \mathcal S(H^p)$, where
$1 < p < \infty$. Recall from Section \ref{Proof2} that the related operator
$f \mapsto \widetilde{C_\phi}(f) = f \circ \phi$ is bounded on the harmonic Hardy space 
$h^p$ for $1 < p < \infty$, and that \eqref{split} holds with respect to  
$h^p = H^p \oplus \overline{H^p_0}$ provided $\phi(0) = 0$.
It follows from  \eqref{split} that   $\widetilde{C_\phi} \in \mathcal S(h^p)$, 
since $\mathcal S(h^p)$ is a linear subspace.
Fix $q$ and $r$ such that $1 < q < p < r < \infty$. 
Since $\widetilde{C_\phi}$ is bounded $h^t \to h^t$ for any 
$t \in (1,\infty)$ and $\widetilde{C_\phi} \in \mathcal S(h^p)$, 
the above extrapolation result  applied to $h^t = L^t(\mathbb T, m)$ yields
that  $\widetilde{C_\phi} \in \mathcal K(h^s)$ for any $q < s < r$. In particular,
$C_\phi \in \mathcal K(H^s)$ for any $q < s < r$ by restricting to 
$H^s \subset h^s$. Hence we have deduced by different means the following 
weak version of Theorem \ref{Main1}: \textit{if  $C_\phi \in \mathcal S(H^p)$,  
then $C_\phi \in \mathcal K(H^p)$  for $1 < p < \infty$}.

\smallskip

Above we do not  address the technical issue 
that \cite{HST} only explicitly deals with real $L^p$-spaces,
whereas the above application requires complex scalars. 
(We are indebted to Francisco Hern\'andez for indicating that 
there is indeed also a complex version.) We leave the above alternative here as an
incomplete digression, because it is  not possible to obtain
the full strength of Theorem \ref{Main1} in this way (cf. the following example).

\begin{example}\label{extra}
We point out for completeness that the extrapolation result  
\cite[Thm. 3.8]{HST} for 
$\mathcal S(L^p) = \mathcal S_p(L^p) \cap \mathcal S_2(L^p)$ 
does not have an analogue 
for  the classes $\mathcal S_p(L^p)$ or $\mathcal S_2(L^p)$. 

In fact, let $(r_n)$ be the sequence of Rademacher functions on $[0,1]$ and 
$f \mapsto Pf = \sum_{n=1}^\infty \langle f, r_n\rangle r_n$ 
the canonical projection $L^p \to M$ for $1 < p < \infty$,
where $M = \overline{\operatorname{span}}\{r_n: n \in \mathbb N\}$.
Since  $M \approx \ell^2$ by the Khinchine inequalities, see e.g. \cite[2.b.3]{LT},
it follows that $P \in \mathcal S_p(L^p)$ by  the total incomparability of $\ell^p$ and $\ell^2$
for $p \neq 2$. 

Furthermore, the results of section \ref{Proof2} 
(in particular, see Example \ref{Ex} and  \eqref{split}) imply 
that for $p \neq 2$ there are composition operators
$\widetilde{C_\phi} \in \mathcal S_2(h^p)$ which fail to be compact.
\end{example}

\subsection{Open problems.}
Our results suggest several natural questions. 

\begin{problems}
(1) Are there results corresponding  to our main theorems 
for $C_\phi \in \mathcal L(H^p,H^q)$ in the case $p \neq q$?
Note that the conditions for boundedness and compactness 
of  $C_\phi: H^p \to H^q$ are different in the respective cases $p < q$ and $q < p$,
and they can be found in \cite{R}, \cite{J} and \cite{Sm}. For instance, for $q < p$
one has that $C_\phi \in \mathcal K(H^p,H^q)$ if and only if \eqref{CM} holds.
On the other hand,  the class $\mathcal S(L^p, L^q)$ also behaves differently
from \eqref{Weis} and \eqref{strict} for $p \neq q$.  For instance, if $p, q > 2$ and $p \neq q$,
then $\mathcal S(L^p, L^q) = \mathcal S_2(L^p, L^q)$ but $\mathcal S_p(L^p, L^q) = 
\mathcal L(L^p, L^q)$. These equalities  follow from the Kadec-Pe{\l}cynski dichotomy 
 \cite[6.4.8]{AK} and  the total incomparability of 
$\ell^p$ and $\ell^q$.

(2) Is there an analogue of Theorem \ref{Main3} for  $p = 1$?

(3) Is the converse of Proposition \ref{BMOA}.(i) also true? 

(4) Is there a Banach space $E$ of 
scalar-valued analytic functions on 
$\mathbb D$ and an analytic map $\phi: \mathbb D \to \mathbb D$, for which
$C_\phi \in  \mathcal S(E) \setminus \mathcal K(E)$? In this direction  
Lefevre et. al. 
 \cite{LLQR} found a non-reflexive
Hardy-Orlicz space $H^\psi$ so that
$C_\phi \in \mathcal W(H^\psi) \setminus \mathcal K(H^\psi)$,  
where $\phi$ is a lens map. 
\end{problems}

The approach sketched in Subsection \ref{alt} suggests that
weaker rigidity properties such as \eqref{rig}
are likely to hold for many other concrete classes of operators on $H^p$.
Subsequently  Miihkinen \cite{M} has used similar techniques as in section \ref{Proof}
to show that the dichotomy of Theorem \ref{Main1} remains valid for the 
class of analytic Volterra operators $T_g$ on $H^p$, where
\[
f \mapsto (T_g(f))(z) = \int_{0}^z f(\tau) g'(\tau) d\tau, \quad z \in \mathbb D.
\]
We refer e.g. to the surveys \cite{A} or \cite{Si} for the conditions 
on the fixed analytic map 
$g: \mathbb D \to \mathbb C$ which characterise 
the boundedness or compactness of $T_g$.

\bibliographystyle{amsalpha}

\end{document}